\documentclass[12pt]{amsart}
\usepackage{latexsym}
\usepackage{a4}
\usepackage{amssymb}
\usepackage{amsbsy}
\newtheorem{thm}{Theorem}[section] 
 
\newtheorem{lem}[thm]{Lemma} 
 
\newtheorem{cor}[thm]{Corollary} 
\theoremstyle{definition}
\newtheorem{exa}[thm]{Example} 
\newtheorem{de}[thm]{Definition}
\numberwithin{equation}{section} 
\newcommand{\ab}[1]{{\mathbf{#1}}}

\newcommand{\N}{\Bbb{ N}} 
\newcommand{\Z}{\Bbb{ Z}} 

\newcommand{\setsuchthat}{\,\, \pmb{|} \,\,}

\newcommand{\vb}[1]{#1}

\newcommand{\Pol}{\mathrm{Pol}}
\newcommand{\Clo}{\mathrm{Clo}}

\newcommand{\Con}{\mathrm{Con}\,}

\newcommand{\A}{\ab{A}}
\newcommand{\B}{\ab{B}}
\newcommand{\C}{\ab{C}}
\newcommand{\D}{\ab{D}}
\newcommand{\epsi}{\varepsilon}
\title{Independence of algebras with edge term}
\author{Erhard Aichinger and Peter Mayr}
\thanks{Supported by the Austrian Science Fund (FWF): P24077 and P24285}
\dedicatory{Dedicated to our teacher G\"unter Pilz on the occasion of his 70th birthday}
\date{March 19, 2015}

\begin{document}
\begin{abstract}
In \cite{Fo:TIOA}, two varieties $V, W$ of the same type are defined
to be \emph{independent} if there is a binary term $t(x,y)$ such that
$V \models t(x,y) \approx x$ and $W \models t(x,y) \approx y$.
In this note, we give necessary and sufficient conditions for two finite algebras with
a Mal'cev term (or, more generally, with an edge term) to
generate independent varieties. 
 In particular we show that the independence of finitely generated varieties with edge term can be
 decided by a polynomial time algorithm.
\end{abstract} 

\maketitle
\section{Independent varieties and algebras}
    In this note, we search for conditions on two varieties
    of the same type to be independent. 
    This notion of independence was introduced in \cite{Fo:TIOA}. Foster
    calls a finite sequence $(V_i)_{i \in \{1, \ldots, n\}}$ of subvarieties of a variety $W$
     \emph{independent}
    if there exists a term $t(x_1, \ldots, x_n)$ such that for each $i \in \{1,\ldots, n\}$, 
     $V_i \models t (x_1, \ldots, x_n) \approx x_i$ \cite[Lemma~2.1]{Fo:TIOA}.
      Gr\"atzer, Lakser, and P\l onka proved that for two independent varieties
     $V_1$ and $V_2$, every algebra in the join $V_1 \vee V_2$ is isomorphic to
     a direct product $\ab{A}_1 \times \ab{A}_2$ with $\ab{A}_1 \in V_1$ and 
     $\ab{A}_2 \in V_2$. It is easy to see that two independent varieties
     $V_1$ and $V_2$ are \emph{disjoint}, meaning that $V_1 \cap V_2$ only contains
     one element algebras. If $V_1$ and $V_2$ are subvarieties of
     a congruence permutable variety, then the converse holds: in fact,
     Hu and Kelenson proved that a sequence $(V_1, \ldots, V_n)$ of subvarieties
     of a congruence permutable variety is independent if and only if $V_i$ and $V_j$ are disjoint
     for all distinct $i,j$
     \cite[Corollary~2.9]{HK:IADF}. Freese and McKenzie showed 
      that if $V_1$ and $V_2$ are disjoint subvarieties of a congruence modular
      such that at least one of the varieties is solvable, then $V_1$ and $V_2$ are independent
      \cite[Theorem~11.3]{FM:CTFC}.
       J\'onsson and Tsinakis proved that the join of two independent finitely
      based varieties
      of finite type is again finitely based \cite[Theorem~3.3]{JT:POCO}; a different finite axiomatization
      of the join is given in \cite[Theorem~3.9]{KPL:OIVA}. In this paper,
       Kowalski, Paoli, and Ledda also gave a characterization
       of independence for disjoint varieties by a Mal'cev-type condition \cite[Theorem~3.2]{KPL:OIVA}. 

       Two algebras $\ab{A}$ and $\ab{B}$ from the same variety are called \emph{independent}
       if they generate independent varieties; this is equivalent to the existence
       of a binary term $t(x,y)$ such that $\ab{A} \models t(x,y) \approx x$ and
       $\ab{B} \models t(x,y) \approx y$. 
      Let $V = V (\ab{A})$, the variety generated by $\ab{A}$, let 
    $W = V (\ab{B})$, and
    let $\ab{F}_V (2)$ and $\ab{F}_W (2)$ be the free algebras in $V$ and $W$ over $2$ generators.
    It is not too hard to see (and will be proved in Lemma~\ref{lem:indep1}) that the following
    condition is equivalent to the independence of $V$ and $W$:
    \begin{quote}
        $\ab{F}_V (2) \times \ab{F}_W (2)$ is the only subdirect product
        of $\ab{F}_V (2) \times \ab{F}_W (2)$.
    \end{quote}
    Hence the independence of $V$ and $W$ can be determined from the subuniverses
    of $\ab{A}^{A^2} \times \ab{B}^{B^2}$. In this note, we will see that for
    finite algebras $\A$ and $\B$, the independence of $V (\A)$ and $V(\B)$ can
    be determined from the subuniverses of $\A^2 \times \B^2$ if $\A$ and $\B$ have
    a common Mal'cev term, and from the subuniverses of $\A^{k-1} \times \B^{k-1}$ if
    $\A$ and $\B$ have a common $k$-edge term with $k \ge 3$.
 From this we obtain a polynomial time algorithm for deciding the independence of two
 finite algebras of finite type with edge term.
     As another application, we obtain a new proof of the description of
     polynomial functions on direct products without skew-congruences from \cite{KM:PFOS}.

\section{Product subalgebras} \label{sec:pro}

In this section, we will describe the shape of subuniverses
of direct products of powers of two algebras.
For a direct product $\ab{E} \times \ab{F}$, we define
$\pi_{\ab{E}} ( e,f ) = e$ and $\pi_{\ab{F}} ( e,f ) = f$ for 
all $e \in E$, $f \in F$.
\begin{de} 
   Let $\ab{E}$ and $\ab{F}$ be two similar algebras.
   We call a subalgebra $\ab{C}$ of $\ab{E} \times \ab{F}$
   a \emph{product subalgebra} if $\ab{C} = \pi_{\ab{E}} (\ab{C}) \times 
      \pi_{\ab{F}} (\ab{C})$.
\end{de}
Hence $\ab{C}$ is a product subalgebra of $\ab{E} \times \ab{F}$ if and only if
for all $(e_1, f_1) \in C$ and $(e_2, f_2) \in C$, we 
have $(e_1, f_2) \in C$. 
%We notice that the notion
%``rectangular'' has a different meaning in 
%\cite{BD:TADT}. 
We note that in this paper, the concept 
of product subalgebras only refers to subalgebras of
direct products of \emph{two} algebras. If we say
that for similar algebras $\ab{A}$ and $\ab{B}$ and
for $m, n \in \N$,
a subalgebra $\ab{C}$ of $\ab{A}^m \times \ab{B}^n$ is 
a product subalgebra, we mean that $(\vb{a}, \vb{b}) \in C$ and
$(\vb{c}, \vb{d}) \in C$ implies $(\vb{a}, \vb{d}) \in C$ for
all $\vb{a},\vb{c} \in A^m$ and $\vb{b},\vb{d} \in B^n$.
We recall that a \emph{tolerance relation} of an algebra
$\ab{A}$
is a subalgebra of $\ab{A} \times \ab{A}$ that
is a reflexive and symmetric relation on $A$. 
\begin{de}
    Let $\ab{A}$ and $\ab{B}$ be similar algebras,
    let $\alpha$ be a subset of $A \times A$, and
     let $\beta$ be a subset of $B \times B$.
    Then the \emph{product} $\alpha \times_c \beta$ is defined by
    \[
       \alpha \times_c \beta =
         \{ ((a_1, b_1), (a_2, b_2)) \in (A \times B) \times (A \times B) \setsuchthat
              (a_1, a_2) \in \alpha, (b_1, b_2) \in \beta \}.
    \]
%    A binary relation $\gamma$ on $A \times B$ is a \emph{product relation}
%    if there are $\alpha \subseteq A \times A$ and $\beta \subseteq B \times B$
%    such that $\gamma = \alpha \times_c \beta$.
 %   
A \emph{product tolerance} of the direct product $\ab{A} \times \ab{B}$ is
a tolerance $\gamma$  on $\ab{A}$ such that
$\gamma = \alpha \times_c \beta$ for some tolerances $\alpha$ of $\ab{A}$ and 
$\beta$ of $\ab{B}$; and $\gamma$ is  \emph{product congruence} of $\ab{A} \times \ab{B}$
if there are $\alpha \in \Con (\ab{A})$ and $\beta \in \Con (\ab{B})$ such that
$\gamma = \alpha \times_c \beta$.
\end{de}
Our main results are the following two theorems.
\begin{thm} \label{thm:1}
   Let $\ab{A}, \ab{B}$ be algebras in a congruence
   permutable variety. We assume that
   \begin{enumerate}
       \item \label{it:ass1} all subalgebras of $\ab{A} \times \ab{B}$ are product subalgebras, and
       \item \label{it:ass2} for all subalgebras $\ab{E}$ of $\ab{A}$ and $\ab{F}$ of $\ab{B}$,
             all congruences of  $\ab{E} \times \ab{F}$ are product congruences.
   \end{enumerate}
   Then for all $m,n \in \N_0$,  all subalgebras of $\ab{A}^m \times \ab{B}^n$ are product subalgebras.
\end{thm}
The proof is given in Section~\ref{sec:proofthm1}. 
We will generalize this result from congruence permutable varieties
to varieties with an edge term \cite{BI:VWFS}.
  Let $k\in\N, k\geq 2$. A $(k+1)$-ary term $t$ in the language of a variety $V$ is a 
 $k$-\emph{edge term} if
\[ t\left(\begin{array}{cccccc} 
   x&x&y&\dots&\dots&y \\
   x&y&x&\ddots&&\vdots \\
   y&y&y&\ddots &\ddots&\vdots \\
%   \vdots &\vdots & \vdots &\ddots &\ddots &\ddots &\vdots  \\
   \vdots &\vdots & \vdots & \ddots & \ddots & y \\
   y&y&y&\dots&y&x \end{array}\right) \approx
\left(\begin{array}{c}
 y \\
 \vdots \\
 \vdots \\
 \vdots \\
 y \end{array}\right). \]
%  For $k\geq 2$, a term $t(x_1,\dots,x_{k+1})$ is a $k$-\emph{edge term} for a variety $V$ if for all
%  $\A\in V$ we have
%  \[ \A\models t(y,y,x,\ldots,x) \approx t(y,x,y,x,\ldots,x) \approx y \]
%  and for all $i \in \{4,\ldots, k+1\}$ we have
% \[ \A\models t(x, \ldots,x, y, x, \ldots, x) \approx x \text{ with } y \text{ in position } i. \]
 We note that a variety has an $2$-edge term if and only if it has a Mal'cev term.
 Every variety with a Mal'cev term or with a near-unanimity term has an edge term.
 The class of algebras with an edge term therefore
 contains all groups and their expansions (such as rings, vector spaces, Lie algebras \ldots),
 all quasigroups, loops, as well as all lattices and their expansions. 

\begin{thm} \label{thm:p}
  Let $k \ge 2$,
  let $\A,\B$ be algebras in a variety with $k$-edge term.
  We assume that
\begin{enumerate}
\item \label{it:small}
 for all $r,s\in\N$ with $r+s\leq\max(2,k-1)$, every subalgebra of $\A^r\times\B^s$
 is a product subalgebra, and
\item \label{it:tol}
 for all subalgebras $\ab{E}$ of $\A$ and $\ab{F}$ of $ \B$, every tolerance of $\ab{E}\times\ab{F}$ is a product tolerance.
\end{enumerate}
 Then for all $m, n \in \N_0$, every subalgebra of $\A^m\times\B^n$ is a product subalgebra.
\end{thm}
The proof is given in Section~\ref{sec:edge}.
In an algebra with a Mal'cev term, all tolerances are congruences.
Hence Theorem~\ref{thm:1} is really a special case of Theorem~\ref{thm:p}.

 In Theorem \ref{thm:indcp}, \ref{thm:indedge}, respectively, we show that finite algebras
 $\ab{A}, \ab{B}$ that satisfy the assumptions of Theorem~\ref{thm:1}, \ref{thm:p}, respectively,
 are actually independent. In Example~\ref{exa:pruefer} we provide examples that show that
 in general independence does not follow for infinite $\ab{A}$ and $\ab{B}$.

\section{Mal'cev algebras} \label{sec:proofthm1}

In this section, we give a proof of Theorem~\ref{thm:1}. 
From a logical point of view, this section could be omitted because
Theorem~\ref{thm:1} is a corollary of Theorem~\ref{thm:p}.
However, we think it is instructive to see the ideas
of the proof first in this case.

\emph{Proof of Theorem~\ref{thm:1}:}
 %  We fix $\ab{A}, \ab{B}$.
 Let $\ab{A}, \ab{B}$ satisfy the assumptions of Theorem~\ref{thm:1}.
       We will prove the claim by showing that for 
       all $m, n \in \N_0$ and 
   for every subalgebra $\ab{C}$ of $\ab{A}^m \times \ab{B}^n$, we have 
    \begin{equation} \label{eq:ind}
        \ab{C} = \pi_{A^m} (\ab{C}) \times \pi_{B^n} (\ab{C}).
    \end{equation}
   We will proceed by induction on $n + m$.
   
   For the induction base, we set $m := 0, n := 0$. The only subalgebra
    of $\ab{A}^0 \times \ab{B}^0$ is clearly a product subalgebra.

% and let $\ab{C}$
%%   be a subalgebra of $\ab{A} \times \ab{B}$. Then 
%   by the assumptions, there are subalgebras $\ab{E}$, $\ab{F}$ of
%    $\ab{A}$ and $\ab{B}$, respectively, such that
%   $\ab{C} = \ab{E} \times \ab{F}$. Since $\pi_{A} (\ab{C}) = \ab{E}$
%   and $\pi_{B} (\ab{C}) = \ab{F}$, we have 
%     $\vb{C} = \pi_{A} (\vb{C}) \times \pi_{B} (\vb{C})$.

   For the induction step, we let $n, m \in \N_0$ be such that $n+m \ge 1$.
   In the case that $m = 0$ or $n=0$, the equality~\eqref{eq:ind}
   clearly holds. Now we assume $n \ge 1$ and $m \ge 1$,
   and we let $\ab{C}$ be a subalgebra of $\ab{A}^m \times \ab{B}^n$.
   We define $\sigma : A^{m} \times B^n \to A^{m-1} \times B^n$ by
   \[
       \sigma ( (a_1,\ldots, a_m), (b_1, \ldots, b_n)) =
               ( (a_1, \ldots, a_{m-1}), (b_1, \ldots, b_n) ),
  \]
  and $\tau : A^m \times B^n \to A^m \times B^{n-1}$ by
  \[
      \tau ( (a_1, \ldots, a_m), (b_1,\ldots, b_n)) =
             ((a_1, \ldots, a_m), (b_1,\ldots, b_{n-1}))
  \]
  for all $a_1,\ldots, a_m \in A, b_1, \ldots, b_n \in B$.

  We are now ready to prove the non-trivial inclusion  $\supseteq$ of 
  \eqref{eq:ind}. To this end,  let $((a_1,\ldots, a_m), (b_1, \ldots, b_n))
   \in \pi_{A^m} (\ab{C}) \times \pi_{B^n} (\ab{C})$.
  Then
   \(  ( (a_1, \ldots, a_{m-1}), (b_1, \ldots, b_{n}) ) \) is an 
   element of $\pi_{A^{m-1}} (\sigma (\ab{C})) \times \pi_{B^{n}} (\sigma (\ab{C}))$.
   Therefore, by the induction hypothesis, we have
   $( (a_1, \ldots, a_{m-1}), (b_1, \ldots, b_{n}) ) \in \sigma (\ab{C})$.
   Thus,
   there is $c \in A$ such that
   \begin{equation} \label{eq:a}
       ( (a_1,\ldots, a_{m-1}, c), (b_1, \ldots b_{n}) ) \in C.
   \end{equation}
    Furthermore, we have
     \[ 
    ( (a_1, \ldots, a_m), (b_1, \ldots, b_{n-1} )) \in
       \pi_{A^m} (\tau (\ab{C})) \times \pi_{B^{n-1}} (\tau (\ab{C})).
   \]
   Therefore, by the induction hypothesis, we have
   $( (a_1, \ldots, a_m), (b_1, \ldots, b_{n-1}) ) \in \tau (\ab{C})$.
   Hence there is $d \in B$ such that
   \begin{equation} \label{eq:b}
       ( (a_1, \ldots, a_m), (b_1, \ldots, b_{n-1}, d) ) \in C.
   \end{equation}

   Next, we define a subset $\alpha$ of $(A \times B)^2$ by
   \begin{multline}
      \alpha := \{ ((x_m, y_n), (x'_m, y'_n)) \setsuchthat \exists 
                   (x_1, \ldots, x_{m-1}) \in A^{m-1}, (y_1, \ldots, y_{n-1}) \in B^{n-1} :  \\
                  ((x_1,\ldots, x_{m-1}, x_m),  (y_1, \ldots, y_{n-1}, y_n)) \in C \text{ and } \\
                  ((x_1,\ldots, x_{m-1}, x'_m), (y_1, \ldots, y_{n-1}, y'_n)) \in C \}.
   \end{multline}
   It is easy to see that $\alpha$ is a reflexive relation
   on $$S := \{ (x_m, y_n) \setsuchthat ((x_1, \ldots, x_m), (y_1, \ldots, y_n)) \in C \}.$$ Furthermore,
   $S$ is a subuniverse of $\ab{A} \times \ab{B}$,
   $\alpha$ is a subuniverse of $(\ab{A} \times \ab{B})^2$, and $\{ (s,s) \setsuchthat s \in S \} \subseteq \alpha \subseteq S^2$.
   Since $\ab{S}$ has a Mal'cev term, this implies that $\alpha \in \Con (\ab{S})$.

   From~\eqref{eq:a} and~\eqref{eq:b}, we obtain $((c, b_n), (a_m, d)) \in \alpha$.
   Since $\ab{S}$ is a product subalgebra of $\ab{A} \times \ab{B}$, we obtain
   $(c,d) \in S$. We will prove next that
   \begin{equation} \label{eq:alpha}
         ((c, b_n), (c, d)) \in \alpha.
   \end{equation}
   All congruences of $\ab{S}$ are product congruences, and therefore,
   there are congruences $\alpha_1 \in \Con (\pi_{\ab{A}} (\ab{S}))$ and
   $\alpha_2 \in \Con (\pi_{\ab{B}} (\ab{S}))$ such that
   $\alpha = \alpha_1 \times_c \alpha_2$.
 Then $((c, b_n), (a_m, d)) \in \alpha$ yields $(b_n, d) \in \alpha_2$.
 Together with $(c,c) \in \alpha_1$, this implies~\eqref{eq:alpha}.  
   
 Hence we have $u_1, \ldots, u_{m-1} \in A$, $v_1, \ldots, v_{n-1} \in B$ such
   that
   \[
      \begin{array}{rcl}
            ((u_1,\ldots, u_{m-1}, c), (v_1, \ldots, v_{n-1}, b_n)) & \in & C, \\
            ((u_1,\ldots, u_{m-1}, c), (v_1, \ldots, v_{n-1}, d)) & \in & C, \\
((a_1,\ldots, a_{m-1}, a_m), (b_1, \ldots, b_{n-1}, d)) & \in & C.
      \end{array}
   \]
 The last line above is equation~\eqref{eq:b}.
%   \[ 
%           ((a_1,\ldots, a_{m-1}, a_m), (b_1, \ldots, b_{n-1}, d)) \in C.
%   \]
   Applying the Mal'cev term to these $3$ lines, we obtain
 $$((a_1,\ldots, a_m), (b_1, \ldots, b_n)) \in C.$$
   This completes the proof of~\eqref{eq:ind}, and hence the induction step. \qed  

\section{Algebras with edge term} \label{sec:edge}
In this section, we will prove Theorem~\ref{thm:p}. To this end,
we need some preparation about algebras with edge term. 
 Let $k\in\N, k\geq 2$. A $(k+3)$-ary term $p$ in the language of a variety $V$ is a 
 $(1,k-1)$-parallelogram term if
\[ p\left(\begin{array}{cccccccc} 
   x&x&y&z&y&\dots&\dots&y \\
   y&x&x&y&\ddots&\ddots&&\vdots \\
   \vdots &\vdots &\vdots & \vdots &\ddots &\ddots &\ddots &\vdots  \\
   \vdots &\vdots &\vdots & \vdots & &\ddots & \ddots & y \\
   y&x&x&y&\dots&\dots&y&z \end{array}\right) \approx
\left(\begin{array}{c}
 y \\
 \vdots \\
 \vdots \\
 \vdots \\
 y \end{array}\right). \]
 A variety has a $k$-edge term iff it has a $(1,k-1)$-parallelogram
 term~\cite[Theorem 3.5]{KS:COAW}. 
 
 We give a slight generalization of representations for subpowers of algebras with edge
 terms from~\cite{BI:VWFS} to subalgebras of direct products.
 For $n\in\N$ and sets $A_1,\dots, A_n$, let $R \subseteq A_1\times\dots\times A_n$.
 For $i \in \{1,\ldots, n\}$, we define the relation $\varphi_i (R)$ on $A_i$ by
\begin{multline*}
    \varphi_i (R) := \{ (a_i, b_i) \in A_i \times A_i  \setsuchthat \\
                        (a_1,\ldots, a_n) \in R,
                        (b_1,\ldots, b_n) \in R,
                        (a_1, \ldots, a_{i-1}) =
                        (b_1, \ldots, b_{i-1}) \}.
\end{multline*}
 An element of $\varphi_i (R)$ is also called a \emph{fork} of $R$ at index $i$.
 If tuples $\vb{a} := (a_1,\ldots, a_n)$ and $\vb{b} := (b_1,\ldots, b_n)$ from $R$ satisfy 
 $(a_1, \ldots, a_{i-1}) = (b_1, \ldots, b_{i-1})$, then we say that $\vb{a},\vb{b}$ \emph{witness}
 the fork $(a_i,b_i)$ at index $i$ in $R$.

 For a tuple $\vb{a} := (a_1,\ldots, a_n)$ and $T\subseteq\{1,\dots,n\}$, let
 $\pi_T(a) := (a_i)_{i\in T}$. 
\begin{de}
 Let $k,n\in\N, k\geq 2$, let $\A_1,\dots,\A_n$ be algebras in a variety with $k$-edge term,
 and let $B\leq \A_1\times\dots\times\A_n$. Then $R\subseteq B$ is a \emph{representation}
 of $\B$ if
\begin{enumerate}
\item $\pi_T(R) = \pi_T(B)$ for all $T\subseteq\{1,\dots,n\}$ with $|T|<k$, and
\item $\varphi_i (R) = \varphi_i (B)$ for all $i \in \{1,\ldots, n\}$.
\end{enumerate}
\end{de}

 The present definition of a representation $R$ differs from the original
 notion~\cite[Definition 3.2]{BI:VWFS} in that it applies to products of algebras not only
 to powers of a single algebra. More importantly, we require witnesses for all forks to be in
 $R$ whereas a representation in the sense of~\cite{BI:VWFS} only needs to contain witnesses for
 forks associated with minority indices. 

\begin{lem}  \label{le:rep}
 Let $n\in\N$, let $\A_1,\dots,\A_n$ be algebras in a variety with $k$-edge term.
 Let $\B$ be a subalgebra of $\A_1\times\dots\times\A_n$ with representation $R$.
 Then $R$ generates~$\B$.
\end{lem}

\begin{proof}
 Let $b\in B$, and let $\langle R\rangle$ denote the subalgebra of $\A_1\times\dots\times\A_n$ that is
 generated by $R$. We will show that
\begin{equation} \label{eq:fb}
 \forall m\in\{1,\dots,n\}\ \exists f\in\langle R\rangle\ \forall i \leq m\colon f_i = b_i
\end{equation}
 by induction on $m$. The result then follows for $m = n$.
 From the definition of a representation~\eqref{eq:fb} holds for $m<k$. Assume $m \geq k$
 in the following. By the induction hypothesis we have $g\in\langle R\rangle$ such that
 $g_i = b_i$ for all $i\leq m-1$. Then $(g_m,b_m)$ is in $\varphi_m(B)$ and hence in
 $\varphi_m(R)$. Hence we have $g',f'\in R$ that witness the fork $(g_m,b_m)$ at $m$.

 We claim that 
\begin{equation} \label{eq:fT}
 \forall\, T\subseteq\{1,\dots,m-1\}\ \exists f^T\in\langle R\rangle\ \forall i\in T\cup\{m\}\colon f^T_i = b_i.
\end{equation}
 We will prove this by induction on $|T|$. Again, for $|T|\leq k-2$, we have such an
 $f^T\in R$ by the definition of a representation. Assume $|T|\geq k-1$ and
 $T = \{i_1,\dots,i_{|T|}\}$. For $j\in\{1,\dots,k-1\}$, we let $U_j := T\setminus\{i_j\}$.
 Now for every $j\in\{1,\dots,k-1\}$ the induction hypothesis yields $f^{U_j}\in\langle R\rangle$
 such that for all $i\in U_j\cup\{m\}$ we have $f^{U_j}_i = b_i$.
 Let $p$ be the $(1,k-1)$-parallelogram term that exists in the variety by~\cite{KS:COAW}.
 We define 
\[ f^T := p(g,g',f',g,f^{U_1},\dots,f^{U_{k-1}}). \]
 For $j\in \{1,\dots,k-1\}$ we obtain
\[ f^T_{i_j} = p(b_{i_j},g'_{i_j},g'_{i_j},b_{i_j},b_{i_j},\dots,b_{i_j},f^{U_j}_{i_j},b_{i_j},\dots,b_{i_j}) = b_{i_j}. \]
 For $i\in T\setminus\{i_1,\dots,i_{k-1}\}$ we have
\[ f^T_i = p(b_i,g'_i,g'_i,b_i,b_i,\dots,b_i) = b_i. \]
 Further
\[ f^T_m = p(g_m,g_m,b_m,g_m,b_m,\dots,b_m) = b_m. \]
 Thus the induction step of~\eqref{eq:fT} is proved. Now~\eqref{eq:fb} follows from~\eqref{eq:fT}
 for $T = \{1,\dots,m-1\}$. 
\end{proof}

\emph{Proof of Theorem~\ref{thm:p}:}
 Let $\ab{A}, \ab{B}$ satisfy the assumptions of Theorem~\ref{thm:p}
 Let $m,n\in\N$, let $\C$ be a subalgebra of $\A^m\times\B^n$, and let
 $D := \pi_{A^m}(C)\times\pi_{B^n}(C)$. We will show that
\begin{equation} \label{eq:CisD}
 C = D
\end{equation}
 by induction on $m+n$.
 The assertion is true if $m = 0$ or $n=0$ or $m+n\leq\max(2,k-1)$ by assumption~\eqref{it:small}.
 So we assume $m > 0, n>0$ and $m+n > \max(2,k-1)$.  
 We consider $\D$ as subalgebra of
 $\underbrace{\A\times\dots\times\A}_{m}\times\underbrace{\B\times\dots\times\B}_{n}$ and claim that
\begin{equation} \label{eq:Crep}
 C \text{ is a representation of } D.
\end{equation} 
 First let $T\subseteq\{1,\dots,m+n\}$ with $|T|<k$. Then $\pi_T(C)$ is a product subalgebra by
 assumption~\eqref{it:small}. It follows that  $\pi_T(C) = \pi_T(D)$.

 Next, by the induction hypothesis, we have $\pi_{\{1,\dots,m+n-1\}}(C) = \pi_{\{1,\dots,m+n-1\}}(D)$.
 Consequently $\varphi_i (C) = \varphi_i (D)$ for all $i \in \{1,\ldots, m+n-1\}$.
 It remains to show $\varphi_{m+n} (C) = \varphi_{m+n} (D)$. The inclusion $\subseteq$ is
 immediate. For the converse consider 
\begin{multline*}
 \gamma := \{ ((f_1,f_{m+n}),(g_1,g_{m+n})) \in (A\times B)^2 \setsuchthat  \\ f,g\in C \text{ and } 
 f_i = g_i \text{ for all } i\in\{2,\dots,m+n-1\} \}. 
\end{multline*}
 Clearly $\gamma$ is a tolerance of $\pi_{\{1,m+n\}}(\C)$, which is a subalgebra of $\A\times\B$.
 By assumption~\eqref{it:small} we have $E\leq\A, F\leq\B$ such that
 $\pi_{1,m+n}(\C) = \ab{E}\times\ab{F}$. 
 By assumption~\eqref{it:tol} we have tolerances $\alpha$ of $\ab{E}$ and $\beta$ of $\ab{F}$
 such that $\gamma = \alpha\times_c\beta$.

 Let $(u,v)\in\varphi_{m+n} (D)$ be a fork that is witnessed by $f,g\in D$.
% such that $f_i = g_i$ for all $i<m+n$ and $f_{m+n} = u, g_{m+n} = v$.
 By the induction hypothesis we have $f',g'\in C$ such that
 $f'_i = f_i$ and $g'_i = g_i$ for all $i > 1$. Then 
\[ ((f'_1,\underbrace{f'_{m+n}}_{u}),(g'_1,\underbrace{g'_{m+n}}_{v})) \in \gamma. \] 
 In particular $(u,v)\in\beta$ and
\[ ((f'_1,u), (f'_1,v)) \in\alpha\times_c\beta = \gamma. \]
 By the definition of $\gamma$ we then have $f'',g''\in C$
 that witness the fork $(u,v)$ at $m+n$.
% such that  $f''_i = g''_i$ for all $i<m+n$ and $f''_{m+n} = u, f''_{m+n} = v$.
 Hence $(u,v)\in\varphi_{m+n} (C)$ and~\eqref{eq:Crep} is proved.
% Thus $C$ is a representation for $D$.
 By Lemma~\ref{le:rep} it follows that $C=D$. \qed

\section{Independent algebras}   \label{sec:independent}
In this section we will relate our results on product subalgebras
to independent varieties and algebras.
The following lemma explains basic relations 
between these concepts.
The implication~\eqref{it:in4}$\Rightarrow$\eqref{it:in1} comes 
from \cite[Corollary~2.9]{HK:IADF} (cf. \cite[Theorem~3.2]{JT:POCO}). 
\begin{lem} \label{lem:indep1}
   Let $\ab{A}$ and $\ab{B}$ be similar algebras. Then the
   following are equivalent:
   \begin{enumerate}
       \item \label{it:in1}  $\ab{A}$ and $\ab{B}$ are independent.
       \item \label{it:in2}  For all (possibly infinite) sets $I, J$, every subalgebra $\ab{E}$ of
             $\ab{A}^I \times \ab{B}^J$ is a product subalgebra.
       \item \label{it:in3}  For all sets $I, J$ with $|I| \le |A|^2$ and $|J| \le |B|^2$,
            every subalgebra $\ab{E}$ of
             $\ab{A}^I \times \ab{B}^J$ is a product subalgebra.
       \item \label{it:in3a} For all subdirect products
                             $\ab{E}$ of $\ab{F}_{V(\ab{A})} (2)$ and $\ab{F}_{V (\ab{B})} (2)$, we have
                             $\ab{E} = \ab{F}_{V(\ab{A})} (2) \times \ab{F}_{V (\ab{B})} (2)$.
   \end{enumerate}
    If $\ab{A}$ and $\ab{B}$ lie in the same congruence permutable 
    variety $V$, then these four items are furthermore equivalent to
    \begin{enumerate}
        \setcounter{enumi}{4}
          \item \label{it:in4} $V(\ab{A})$ and $V(\ab{B})$ are disjoint.
    \end{enumerate} 
\end{lem}
\emph{Proof:}
 \eqref{it:in1}$\Rightarrow$\eqref{it:in2}:
    Let $t$ be a binary term witnessing the independence of $\ab{A}$ and $\ab{B}$.
    Let $(\vb{a}, \vb{b})$ and $(\vb{c}, \vb{d})$ be elements
    of $\ab{E}$. Then $t^{\ab{A}^I \times \ab{B}^J} ( (\vb{a}, \vb{b}), (\vb{c}, \vb{d}) ) 
    = (\vb{a}, \vb{d})$, and hence $(\vb{a}, \vb{d}) \in E$.

 \eqref{it:in2}$\Rightarrow$\eqref{it:in3}: Obvious.

 \eqref{it:in3}$\Rightarrow$\eqref{it:in3a}: Let $\ab{E}$ be a subdirect product of
    $\ab{F}_{V(\ab{A})} (2) \times \ab{F}_{V (\ab{B})} (2)$. $\ab{F}_{V (\ab{A})} (2)$ is isomorphic
    to a subalgebra $\ab{A'}$ of $\ab{A}^{A^2}$, and similarly $\ab{F}_{V (\B)} (2) \cong \ab{B'} \le \ab{B}^{B^2}$.
    Via these isomorphisms, we obtain an isomorphic copy $\ab{E'}$ of $\ab{E}$ such that
    $\ab{E'}$ is a subalgebra of $\ab{A}^{A^2} \times \ab{B}^{B^2}$. Using item~\eqref{it:in3} and the fact
    that $\ab{E'}$ is a subdirect product of $\ab{A'}$ and $\ab{B'}$,
    we obtain $\ab{E'} = \ab{A'} \times \ab{B'}$, and thus $\ab{E} = \ab{F}_{V(\ab{A})} (2) \times \ab{F}_{V (\ab{B})} (2)$.
     
 \eqref{it:in3a}$\Rightarrow$\eqref{it:in1}:
    We let $F_{V (\A)} (2) = F_{V (\A)} (x,y)$, and for two terms $s(x,y)$ and $t(x,y)$, we write
    $s {\sim}_{\A} t$ if $\A \models s \approx t$.
    Now $E := \{ (s / {\sim}_{\A}, s / {\sim}_{\B}) \setsuchthat s \text{ is a term in $x, y$} \}$
    is the universe of a subdirect product of $F_{V (\A)} (x,y) \times F_{V (\B)} (x,y)$. 
    By~\eqref{it:in3a}, $(x / {\sim}_{\A}, y / {\sim}_{\B} ) \in E$, and therefore there exists
    a binary term $t(x,y)$ such that $t {\sim}_{\A} x$ and $t {\sim}_{\B} y$. Thus
    $\A$ and $\B$ are independent.

  \eqref{it:in1}$\Rightarrow$\eqref{it:in4}: Let $t(x,y)$ be a binary term
     witnessing the independence of $\ab{A}$ and $\ab{B}$.
     Then
     $V(\ab{A}) \cap V (\ab{B}) \models x \approx t(x,y) \approx y$, and
      thus this intersection contains only one element algebras.

  For the implication \eqref{it:in4}$\Rightarrow$\eqref{it:in2}, we assume
  that $\ab{A}$ and $\ab{B}$ lie in a congruence permutable variety.
  Let $I$ and $J$ be sets, and let $\ab{E}$ be a subalgebra of 
  $\ab{A}^I \times \ab{B}^J$. Then $\ab{E}$ is a subdirect product
  of $\pi_{\ab{A}^I} (\ab{E}) \times \pi_{\ab{B}^J} (\ab{E})$. Now
  by Fleischer's Lemma \cite[Lemma~IV.10.1]{BS:ACIU}, there is 
  an algebra $\ab{D}$ and there are surjective homomorphisms
  $\alpha' : \pi_{\ab{A}^I} (\ab{E}) \to \ab{D}$ and  
  $\beta'  : \pi_{\ab{B}^J} (\ab{E}) \to \ab{D}$ such that
  $E = \{ (x,y) \in \pi_{\ab{A}^I} (\ab{E}) \times \pi_{\ab{B}^J} (\ab{E})
                   \setsuchthat 
                 \alpha' (x) = \beta' (y) \}$. Since $\ab{D}$ lies
   in $V (\ab{A}) \cap V (\ab{B})$, we have $|D| = 1$. Thus
   $\ab{E} = \pi_{\ab{A}^I} (\ab{E}) \times \pi_{\ab{B}^J} (\ab{E})$, and it is
   therefore a product subalgebra. \qed

Now our results from Section~\ref{sec:pro} give the following
 characterizations of independence:
% sufficient conditions for two finite algebras to be independent:
\begin{thm} \label{thm:indcp}
 Let $\ab{A}$, $\ab{B}$ be finite algebras in a congruence permutable variety.
% Then each of the following
%   conditions implies that $\ab{A}$ and $\ab{B}$ are independent.
 Then the following are equivalent:
    \begin{enumerate}
\item \label{it:sc0}
  $\ab{A}$ and $\ab{B}$ are independent.
       \item \label{it:sc1}
             All subalgebras of $\ab{A} \times \ab{B}$ are product subalgebras,
             and all congruences of all subalgebras of $\ab{A} \times \ab{B}$ are
             product congruences.
       \item \label{it:sc2}
             All subalgebras of $\ab{A}^2 \times \ab{B}^2$ are product subalgebras.
       \item \label{it:sc3}
             $HS (\ab{A}^2) \cap HS (\ab{B}^2)$ contains only one element algebras.
    \end{enumerate}
\end{thm}  

\begin{proof}
 \eqref{it:sc1}$\Rightarrow$\eqref{it:sc0}:
%  We start by proving that~\eqref{it:sc1} implies that $\ab{A}$ and $\ab{B}$ are independent.
  We assume that~\eqref{it:sc1} holds. Then by Theorem~\ref{thm:1}, for
   all $m, n \in \N_0$ with   $m \le  |A|^2$ and  $n := |B|^2$, all subalgebras
   of $\ab{A}^m \times \ab{B}^n$ are product subalgebras. Now by Lemma~\ref{lem:indep1},
   $\ab{A}$ and $\ab{B}$ are independent.

\eqref{it:sc2}$\Rightarrow$\eqref{it:sc1}:
%   Next, we prove~\eqref{it:sc2}$\Rightarrow$\eqref{it:sc1}.
% To this end, let
 Let $\gamma$ be a congruence of $\ab{A} \times \ab{B}$. 
   Let $\ab{E} (\gamma)$ be the subalgebra of $\ab{A}^2 \times \ab{B}^2$ given by
    \[
       E(\gamma) = \{ ( (a_1, a_2), (b_1, b_2)) \setsuchthat
                        ((a_1, b_1), (a_2, b_2)) \in \gamma \}.
    \]
    From the fact that $\ab{E} (\gamma)$ is a product subalgebra,
    we obtain that $\gamma$ is a product congruence of $\A \times \B$.

\eqref{it:sc3}$\Rightarrow$\eqref{it:sc2}:
%    For \eqref{it:sc3}$\Rightarrow$\eqref{it:sc2},
 We let $\ab{C}$ be a subalgebra
    of $\ab{A}^2 \times \ab{B}^2$. By Fleischer's Lemma, there 
    are a subalgebra $\ab{A}'$ of $\ab{A}^2$, a subalgebra
    $\ab{B}'$ of $\ab{B}^2$, an algebra $\ab{D}$, and epimorphisms
    $\alpha' : \ab{A}' \to \ab{D}$, $\beta' : \ab{B}' \to \ab{D}$
    such that $C = \{ (a', b') \in A' \times B' \setsuchthat
     \alpha' (a') = \beta' (b') \}$. Since $\ab{D} \in HS (\ab{A}^2) \cap
     HS (\ab{B}^2)$, $\ab{D}$ is a one element algebra, and therefore
     $C = A' \times B'$, and it is therefore a product subalgebra.
%    This completes the proof of the implication \eqref{it:sc3}$\Rightarrow$\eqref{it:sc2}.
%    Hence algebras satisfying \eqref{it:sc3} are independent.

\eqref{it:sc0}$\Rightarrow$\eqref{it:sc3} is proved in the same way
 as~\eqref{it:in1}$\Rightarrow$\eqref{it:in4} of Lemma~\ref{lem:indep1}.
\end{proof}

\begin{thm} \label{thm:indedge}
 Let $k\geq 2$, and let $\ab{A}$, $\ab{B}$ be finite algebras in a variety with $k$-edge term.
% Then each of the following
%   conditions implies that $\ab{A}$ and $\ab{B}$ are independent.
 Then the  following are equivalent:
\begin{enumerate}
\item \label{it:sc0a}
  $\ab{A}$ and $\ab{B}$ are independent.
       \item \label{it:sc4}
 For all $r,s\in\N$ with $r+s\leq\max(2,k-1)$, every subalgebra of $\A^r\times\B^s$
 is a product subalgebra, and for all $E\leq\A, F\leq\B$, every tolerance of $\ab{E}\times\ab{F}$ is a 
             product tolerance.
       \item \label{it:sc5} 
 For all $r,s\in\N$ with $r+s\leq\max(4,k-1)$, every subalgebra of $\A^r\times\B^s$
             is a product subalgebra.
     \end{enumerate}
\end{thm}

\begin{proof}
\eqref{it:sc4}$\Rightarrow$\eqref{it:sc0a} follows from Theorem~\ref{thm:p} and
     the implication
     \eqref{it:in3}$\Rightarrow$\eqref{it:in1} of 
     Lemma~\ref{lem:indep1}. 
%     Assume that \eqref{it:sc4} holds. Then Theorem~\ref{thm:p} and
%     the implication
%     \eqref{it:in3}$\Rightarrow$\eqref{it:in1} of 
%     Lemma~\ref{lem:indep1} 
%      imply that 
%     $\ab{A}$ 
%     and $\B$ are independent.

\eqref{it:sc5}$\Rightarrow$\eqref{it:sc4},
%     For \eqref{it:sc5}$\Rightarrow$\eqref{it:sc4},
 We associate
     a subalgebra $\ab{E} (\tau)$ of $\ab{A}^2 \times \ab{B}^2$ with every tolerance
     $\tau$ of $\ab{A} \times \ab{B}$ and proceed as in the proof
     of \eqref{it:sc2}$\Rightarrow$\eqref{it:sc1} of Theorem~\ref{thm:indcp}.

\eqref{it:sc0a}$\Rightarrow$\eqref{it:sc5} follows from~\eqref{it:in1}$\Rightarrow$\eqref{it:in2}
 of Lemma~\ref{lem:indep1}.
\end{proof}

\section{Polynomial functions}
    As observed in \cite{Fo:TIOA, GLP:JADP}, term functions of independent algebras
    $\A$ and $\B$
    can be \emph{paired} in the sense that for all $k$-ary terms $r$ and $s$,
     the term $u := t (r,s)$ satisfies $\A \models u \approx r$ and $\B \models u \approx s$,
    where $t$ is a binary term witnessing the independence of $\A$ and $\B$.
    Hence we have:
        \begin{lem} \label{lem:terms}
        Let $\ab{A}$ and $\B$ be similar independent algebras, and let $k \in \N$.
        Then the mapping
        $\phi : \Clo_k (\A) \times \Clo_k (\B) \to \Clo_k (\A \times \B)$ with
        $$\phi (f,g) ( (a_1,b_1), \ldots, (a_k, b_k)) = (f (\vb{a}), g(\vb{b}))$$
% \text{ for }
 for $f\in\Clo_k(\A),g\in\Clo_k (\B), \vb{a} \in A^k, \vb{b} \in B^k$
        is a bijection.
        \end{lem}
    \emph{Proof:} 
       We first show that for $f \in \Clo_k (\ab{A})$ and $g \in \Clo_k (\ab{B})$,
       we have $\phi (f,g) \in \Clo_k (\ab{A} \times \ab{B})$.
       Let $r$ and $s$ be $k$-variable terms with $r^{\ab{A}} = f$ and
       $s^{\B} = g$. 
       Let $t$ be a term witnessing the independence of $\A$ and $\B$.
       Then for $u := t(r,s)$ we have
       $u^{\A \times \B} ((a_1,b_1), \ldots, (a_k, b_k)) =
        (r^{\A} (\vb{a}), s^{\B} (\vb{b}))$ for all $\vb{a} \in A^k$ and $\vb{b} \in \B^k$.
       The mapping $\phi$ is clearly injective, and for proving that $\phi$ is surjective, let $u$ be a term. Then
       $\phi (u^{\A}, u^{\B}) = u^{\A \times \B}$, and hence the range of $\phi$ 
       contains $\Clo_k (\A \times \B)$. \qed
    
     We can now easily provide an example showing that Theorem~\ref{thm:indcp} does not hold
     for infinite algebras.

\begin{exa} \label{exa:pruefer}
 Let $p$ and $q$ be different primes, and let $\ab{A}$ be the Pr\"ufer
     group $\Z_{p^{\infty}}$ and $\B := \Z_{q^{\infty}}$. It is easy to see that
     for all $m,n \in \N$, all subalgebras of $\ab{A}^m \times \ab{B}^n$ are product subalgebras.
     Since all binary  term functions of $\ab{A} \times \ab{B}$ are of the form $((a_1, b_1), (a_2, b_2))
     \mapsto (z_1 a_1 + z_2 a_2, z_1 b_1 + z_2 b_2)$, there is no term inducing the
      function $((a_1, b_1), (a_2, b_2)) \mapsto (a_1, b_2)$. Hence $\A$ and $\B$ are not independent.
\end{exa}
     
     We will now consider polynomial functions on direct products of two algebras.
     Pilz conjectured that for expanded groups $\ab{A}$ and $\ab{B}$ such that
     all congruences of $\ab{A} \times \ab{B}$ are product congruences the following holds:
     for all unary polynomial functions $f$ on $\A$ and $g$ on $\B$, the function
     $(a,b) \mapsto (f(a), g(b))$ is a polynomial function on $\A \times \B$
     (cf. \cite[Conjecture~2.10]{Pi:NOCF}). In \cite{Ai:ONIA}, the conjecture
     was verified for finite $\A$ and $\B$, and \cite{KM:PFOS} generalized
     this result to finite algebras with a Mal'cev or a majority term.
     The following theorem generalizes two of their results 
     to algebras
     with a $3$-edge term.
     \begin{thm} \label{thm:polies}
        Let $\ab{A}$ and $\ab{B}$ be finite algebras in a variety
        with a $3$-edge term, and let $k \in \N$.
        We assume that every tolerance of $\ab{A} \times \ab{B}$
        is a product tolerance.
        Let
                 $\psi : \Pol_k (\ab{A}) \times \Pol_k (\ab{B}) \to
                          (A \times B)^{(A \times B)^k}$ be the mapping
          defined by 
                 \[ \psi (f, g) \, ( (a_1, b_1), \ldots, (a_k, b_k) ) :=
                      (f (\vb{a}), g (\vb{b}))\]
        for $f\in\Pol_k (\ab{A}),g\in\Pol_k (\ab{B})$, $\vb{a} \in A^k$, and $\vb{b} \in B^k$.
 Then $\psi$ is a bijection from
                     $\Pol_k (\ab{A}) \times \Pol_k (\ab{B})$ to
                     $\Pol_k (\ab{A} \times \B)$.
     \end{thm}
    \emph{Proof:} 
          For each $a \in A, b \in B$, we add a constant operation $c_{(a,b)}$
           to our language.
          By $\ab{A}^*$, we denote the expansion of $\A$ satisfying
           $c_{(a,b)}^\A () = a$, and we write  $\B^*$ for the expansion of
          $\B$ with $c^\B_{(a,b)} () = b$ for all $a \in A, b \in B$.
          It is easy to see that $\Clo_k (\A^*) = \Pol_k (\A)$,
          $\Clo_k (\B^*) = \Pol_k (\B)$, and $\Clo_k (\A^* \times \B^*) =
           \Pol_k (\A \times \B)$. By its construction, $\A^* \times \B^*$ has 
          no proper subuniverses. Since $\A \times \B$ and
         $\A^* \times \B^*$ have the same tolerances, every tolerance
          of $\A^* \times \B^*$ is a product tolerance. 
           We apply Theorem~\ref{thm:p} and obtain
           that $\A^*$ and $\B^*$ are independent.
           Now the result follows from Lemma~\ref{lem:terms}. \qed 
                
    As a corollary, we obtain Corollary~2 and Theorem~3 of \cite{KM:PFOS}.
    \begin{cor}[\cite{KM:PFOS}]
         Let $\ab{A}$ and  $\ab{B}$ be algebras in the variety $V$, and let $k \in \N$.
         If either
         \begin{enumerate}
            \item $V$ has a majority term, or
            \item $V$ is congruence permutable, and every congruence of $\ab{A} \times \ab{B}$
                  is a product congruence,
          \end{enumerate}
         then for all polynomial functions $f \in \Pol_k (\A)$ and 
         $g \in \Pol_k (\B)$, there is a polynomial function $h \in \Pol_k (\A \times \B)$
         with $h ( (a_1, b_1), \ldots, (a_k, b_k)) = (f (\vb{a}), g (\vb{b}))$ for  
         all $\vb{a} \in A^k$ and $\vb{b} \in B^k$.
   \end{cor}
   \emph{Proof:}
      Let us first assume that $V$ has a majority term $m (x,y,z)$. 
      Then the term $e(x_1, x_2, x_3, x_4) := m (x_2, x_3, x_4)$
      is a $3$-edge term. All tolerances of $\A \times \B$
      are product tolerances. To see this, let $\epsi$ be a tolerance
      of $\A \times \B$, and let $( (a_1,b_1), (a_2, b_2)) \in \epsi$ and 
      $( (a_3, b_3), (a_4, b_4) ) \in \epsi$.
      Writing $m$ for the induced term operation $m^{(\A\times\B)\times(\A\times\B)}$ we obtain
      \begin{multline*}
        m \bigl(  m ( ((a_1,b_1), (a_2,b_2)), ((a_1, b_3), (a_1, b_3)), ((a_2, b_3), (a_2, b_3)) ), \\
           m ( ((a_1,b_1), (a_2,b_2)), ((a_1, b_4), (a_1, b_4)), ((a_2, b_4), (a_2, b_4)) ), \\
           m ( ((a_2,b_3), (a_2,b_3)), ((a_2, b_4), (a_2, b_4)), ((a_3, b_3), (a_4, b_4)) )\bigr) \\ = ((a_1, b_3), (a_2, b_4)) \in
        \epsi.
      \end{multline*}  
 Thus $\epsi$ is a product tolerance.
      Now Theorem~\ref{thm:polies} yields the required polynomial $h$.
      In the case that $V$ has a Mal'cev term $d$, then $e (x_1, x_2, x_3, x_4) := d (x_2, x_1, x_3)$
      is a $3$-edge term, and all tolerances of algebras in $V$ are congruences.
      Hence all tolerances of $\A \times \B$ are product tolerances.
      The result follows from Theorem~\ref{thm:polies}.
 \qed

\section{Algorithms}

 Let $\A,\B$ be finite algebras of fixed finite type. For $i\in\{1,2\}$ let
 $e_i\colon A^2\to A, (x_1,x_2)\mapsto x_i,$ and $f_i\colon B^2\to B, (x_1,x_2)\mapsto x_i,$ denote the $i$-th
 projection on $A^2$ and $B^2$, respectively. Then $e_i\in A^{A^2}$ and $f_i\in B^{B^2}$.
 From the definition, $\A$ and $\B$ are independent iff there exists a binary term operation
 on $\A\times\B$ that is $e_1$ on the factor $A$ and $f_2$ on the factor $B$.
 Equivalently, $\A$ and $\B$ are independent iff
\[ (e_1,f_2) \text{ lies in the subalgebra of } \A^{A^2} \times\B^{B^2} \text{ that is generated by } (e_1,f_1),(e_2,f_2). \]
 Using a straightforward closure algorithm that enumerates all elements of the generated algebra,
 the last condition can be checked in time exponential in $\max(|A|,|B|)$.
 Hence deciding independence of two arbitrary finite algebras is in EXPTIME.
 The results from Section~\ref{sec:independent} yield easy polynomial time algorithms for
 algebras with a Mal'cev term or more generally an edge term.

\begin{thm} \label{thm:algedge}
 Let $k\geq 2$, and let $\ab{A}$, $\ab{B}$ be finite algebras in a variety of finite type
 with $k$-edge term.
 Then the independence of $\ab{A}$ and $\ab{B}$ can be decided in time polynomial in $\max(|A|,|B|)$.
\end{thm}

\begin{proof}
 Let $m := \max(4,k-1)$ and $n := \max(|A|,|B|)$. 
 By Theorem~\ref{thm:indedge}~\eqref{it:sc5} the algebras $\ab{A}$ and $\ab{B}$ are independent iff
 for all $r,s\in\N$ with $r+s\leq m$ and for all $a,c\in A^r, b,d\in B^s$
\[ (a,d)\in\langle (a,b),(c,d)\rangle. \]
 To verify the latter condition we need to solve at most $(m-1)n^{2m}$ membership problems 
 for algebras of size at most $n^m$. Each of these membership problems can be decided
 using a closure algorithm in time polynomial in $n$ with the actual degree of
 the polynomial depending on $m$ and the arities of the basic operations of the algebras. 
\end{proof}

\section{Acknowledgments}
 The authors thank K.\ Kaarli, O.\ Koshik, J.\ Opr\v{s}al, and M.\ Volkov for helpful discussions
 on this topic.

\bibliographystyle{plain}
%\bibliography{dp11bib}
\def\cprime{$'$}

\begin{flushleft}
\begin{small}
Institute for Algebra, Johannes Kepler University Linz,
Austria \\
{\tt erhard@algebra.uni-linz.ac.at}\\
{\tt peter.mayr@jku.at}
\end{small}
\end{flushleft}

\end{document}